\documentclass[a4paper,10pt]{article}
 \usepackage{a4wide}

\usepackage[utf8]{inputenc}

\usepackage{amssymb, amsmath, amsfonts, amssymb, mathtools, amsthm}
\usepackage{hyperref}
\usepackage{xcolor}
\usepackage{cite}

\usepackage{enumerate}

\usepackage{authblk}

\newcommand{\err}{\varepsilon}
\newcommand{\lip}{c} 
\newcommand{\const}{c} 

\newcommand{\defeq}{\vcentcolon=}

\renewcommand{\le}{\leqslant}
\renewcommand{\ge}{\geqslant}
\renewcommand{\leq}{\leqslant} 
\renewcommand{\geq}{\geqslant}

\renewcommand{\emptyset}{\varnothing}

\newcommand{\parfrac}[2]{\frac{\partial{#1}}{\partial{#2}}}

\newcommand{\R}{\mathbb{R}}
\newcommand{\N}{\mathbb{N}}

\newcommand{\dd}{\mathrm{d}}
\newcommand{\diff}{\mathrm{d}}

\newcommand{\xhat}{\hat{x}}
\newcommand{\eps}{\varepsilon}
\renewcommand{\epsilon}{\varepsilon}
\newcommand{\RR}{\mathsf{R}}

\newcommand{\RRR}{\mathcal{R}}
\renewcommand{\O}{\mathrm{O}}
\newcommand{\dt}{\Delta} 
\newcommand{\Obs}{\mathcal{O}}
\newcommand{\KK}{\mathcal{K}}
\newcommand{\sym}{\mathrm{S}^n_{++}}
\newcommand{\e}{\mathrm{e}}

\newcommand{\DD}{\mathcal{D}}

\newcommand{\lambdamax}{\bar{\lambda}} 
\renewcommand{\mid}{\colon}

\DeclareMathOperator{\sat}{sat}
\DeclareMathOperator{\Id}{Id}
\DeclareMathOperator{\Gram}{\mathcal{G}} 

\newtheorem{theorem}{Theorem}[section]
\newtheorem{assumption}[theorem]{Assumption}
\newtheorem{definition}[theorem]{Definition}

\newtheorem{remark}[theorem]{Remark}

\newcommand{\card}[1]{|#1|}

\title{Output feedback stabilization of polynomial state-affine control systems 
using control templates
}
\author[1]{Ludovic Sacchelli}
\author[2]{Lucas Brivadis}
\author[3]{Ulysse Serres}
\author[4]{Ita{\"i} Ben Yaacov}

\affil[1]{\small Inria, Université Côte d’Azur, CNRS, LJAD, MCTAO team, Sophia Antipolis, France
 (email: \texttt{ludovic.sacchelli@inria.fr})}

\affil[2]{\small Université Paris-Saclay, CNRS, CentraleSupélec, Laboratoire des Signaux et Systèmes, 91190, Gif-sur-Yvette, France
(email: \texttt{lucas.brivadis@centralesupelec.fr})
}

\affil[3]{\small Universit\'e Claude Bernard Lyon 1, CNRS, LAGEPP UMR 5007, 43 bd du 11 novembre 1918, F-69100 Villeurbanne, France (emails: \texttt{ulysse.serres@univ-lyon1.fr})}

\affil[4]{\small Université Claude Bernard Lyon 1, Institut Camille Jordan, CNRS UMR 5208,
	43 boulevard du 11 novembre 1918, F-69100 Villeurbanne, France (email: \texttt{begnac@math.univ-lyon1.fr})}

\date{\today}

\begin{document}

\maketitle

\begin{abstract}

We establish a separation principle for the output feedback stabilization of state-affine systems that are observable at the stabilization target.
Relying on control templates (recently introduced in \cite{andrieu:hal-04387614}),
that allow to approximate a feedback control while maintaining observability, we design a closed-loop hybrid state-observer system that we show to be semi-globally asymptotically stable. Under assumption of polynomiality of the system with respect to the control, we give an explicit construction of control templates. We illustrate the results of the paper with numerical simulations.
\end{abstract}

\section{Introduction}
The establishment of a separation principle,
being able to achieve state estimation and state feedback stabilization conjointly allows to achieve semi-global output feedback stabilization, was achieved in the 90s under strict observability constraints. Namely, for \cite{TeelPraly1994} and \cite{jouan1996finite},
the separation principle was obtained under a uniform observability assumption, according to which the system remains observable for any possible input.
Although this approach is successful, it discards the fact that most nonlinear systems are not actually uniformly observable.
For instance, in the bilinear case, systems that are not observable for some constant inputs form an open and dense subset
(see, e.g., \cite{brivadis2023output}). 
Lifting the uniform observability assumption requires to mitigate observability loss in control strategies.
In that regard, hybrid strategies have been fairly successful. For instance, \cite{shim2003asymptotic} proposes a periodic switching strategy between an observable input and a feedback law, allowing for practical stabilization. Other examples include \cite{nesic,coron1994stabilization}, as well as \cite{9363594, preprintMPLLAR} in application of anti-lock braking systems, and recent works by the authors of the present paper \cite{brivadis:hal-03845706,brivadis:hal-03180479}.

Here,
we discuss a technique to maintain observability for
state-affine
systems that are not uniformly observable.
Recently introduced in \cite{andrieu:hal-04387614} for analytic nonlinear systems, control templates are inputs that can be scaled and rotated while guaranteeing observability. When a control template exists, 
it can be paired with a periodically sampled feedback law to design a stabilizing time varying control.
In that way, this technique can be understood as a generalization of sample-and-hold control (see, e.g., \cite{1024472} and references therein).
This strategy has been successfully applied in \cite{andrieu:hal-04387614} to prove a separation principle for analytic systems paired with a high-gain observer.
Also,
as shown in \cite{andrieu:hal-04387614}
in the context of analytic systems,
control templates are generic among analytic inputs.
However,
this result is primarily qualitative, lacking methods to exhibit control templates or verify whether an input is a template, which is a current drawback of the theory.
Even for bilinear systems, checking that a control is a template is not immediate.
For instance, since most bilinear systems admit constant inputs for which they are  not observable, a constant input is not, in general, a control template.

This paper contains two main contributions
that address limitations of \cite{andrieu:hal-04387614}.
Firstly, motivated by bilinear systems, we provide an algorithmic construction of control templates for polynomial state-affine systems, relying on
algebraic insights from the theory
of multivariate polynomials. This makes the control templates method explicit.
Secondly, 
assuming only observability at the target (instead of uniform observability \cite{182484}) and existence of a stabilizing state feedback,
we develop a separation principle for state-affine systems.
To do so, we employ the control template strategy
and use a Kalman-like observer (for which state-affine systems are well-suited) that remains in the original coordinates of the system, contrary to high-gain observers. This also allows to consider discontinuous control templates.

\smallskip
{\it Notation}.
We denote by $|\cdot|$ the Euclidean norm on $\R^n$.
The transposition operation is denoted by $'$.
We write
$\O(p)$ for
the group of $p\times p$ real orthogonal matrices and $\sym$ for
the set of $n\times n$ real positive-definite symmetric matrices.

\section{Problem statement  and preliminaries}

\subsection{System stabilization}
Consider a state-affine control system of the form:
\begin{equation}\label{syst}
\left\{
\begin{aligned}
&\dot x = A(u)x + b(u)
\\
&y = C(u)x
\end{aligned}
\right.
\end{equation}
where $x\in\R^n$ is the state, $y\in\R^m$ is the measured output, $u\in \R^p$ is the control input,
and where $A(u)\in\R^{n\times n}$, $b(u)\in\R^{n}$ and $C(u)\in\R^{m\times n}$.
In this paper,
we make the following assumption.
\begin{assumption}\label{ass:pol}
    The mapping $u\mapsto(C(u), A(u))$ is polynomial
    and $u\mapsto b(u)$ is continuous.
\end{assumption}

This condition includes the case of bilinear systems, namely state-affine systems where $A$ and $b$ are affine in $u$, while $C$ is constant.
According to the Cauchy-Lipschitz theorem,
for each locally essentially bounded input
$u\in L^\infty_\mathrm{loc}(\R_+, \R^p)$
and for each initial condition $x_0\in\R^n$,
the Cauchy problem associated to \eqref{syst} admits a unique maximal solution.

In this paper, we discuss output feedback stabilization from the point of view that proper state feedback stabilization is achievable and can be exploited to obtain a separation principle. For this reason, we make the assumption that a feedback law is already given.
\begin{assumption}\label{ass:stab}
    There exists a locally Lipschitz continuous feedback law $\lambda:\R^n\to\R^p$
    such that \eqref{syst}
    in closed-loop with $u = \lambda(x)$
    is
    locally exponentially stable (LES) at the origin.
Without loss of generality, we assume that $\lambda(0) = 0$.
\end{assumption}

\subsection{Observability}\label{SS:obs}

Our goal is to address the problem of semi-global dynamic output feedback stabilization of \eqref{syst}, while avoiding any uniform observability assumption. 
In the context of state-affine systems, the usual notion of observability can be adequately replaced with the equivalent notion of positive-definiteness of the observability Gramian.
For a given locally essentially bounded input $u:\R_+\to \R^p$,
the (backward) observability Gramian over $[s, t]$ is the positive semi-definite symmetric matrix given by
\begin{equation}\label{E:def_gram}
    \Gram_u(t, s)
    = \int_{s}^{t}\Phi_u(\tau, t)'C(u(\tau))'C(u(\tau))\Phi_u(\tau, t)\dd \tau,
\end{equation}
where $\Phi_u(t, s)$ is the transition matrix solution to
\[
\frac{\partial \Phi_u}{\partial t}(t, s)
=A(u(t))\Phi_u(t, s),\quad \Phi_u(s, s) = \Id.
\]
The system is said to be \emph{observable} for the input $u$ over $[s,t]$ when $\Gram_u(t,s)$ is positive-definite 
while \emph{inobservability} corresponds to the kernel of $\Gram_u(t,s)$ being non-trivial.
This is completely natural since for any two trajectories $x,\tilde x$ of \eqref{syst}, with same input $u$ and of respective outputs $y,\tilde y$,
$$
\int_s^t|y(\tau)-\tilde y(\tau)|^2\diff \tau=(x(t)-\tilde x(t))'\Gram_u(t,s)(x(t)-\tilde x(t)).
$$
Moreover,
the observability Gramian satisfies, for $r\leq s\leq t$,
\begin{equation}
\Gram_u(t, r)
    =\Phi_u(s, t)'\Gram_u(s, r)\Phi_u(s, t) +\Gram_u(t, s).
    \label{eq:gram_chasles}
\end{equation}
As a consequence, we recover the usual implication that observability over a given interval implies observability over any encompassing interval.

Recall that in the case of linear time-invariant systems, observability is determined by Kalman's test. Therefore, for a constant input $u\in \mathbb{R}^p$, observability is equivalent to \emph{the pair $(C(u), A(u))$ being observable}, meaning full rank of the $mn \times n$ Kalman observability matrix defined as
\[
\Obs(u)
=\begin{pmatrix}
C(u) \\ C(u)A(u) \\ \vdots \\ C(u)A(u)^{n-1}
\end{pmatrix}.
\]

As mentioned in the introduction, uniform observability, meaning $\Gram_u(t, s) > 0$ for all inputs $u$, is typically
required for output feedback stabilization.
In this paper,
we relax this assumption to the following.

\smallskip
\begin{assumption}\label{ass:0obs}
    The pair $(C(0), A(0))$ is observable.
\end{assumption}

\smallskip
\textit{Problem statement.}
Under Assumptions~\ref{ass:pol}, \ref{ass:stab} and \ref{ass:0obs}, design a hybrid closed-loop output feedback stabilization scheme based on the coupling of the system with an observer. Since the observability assumptions are quite low, we rely on control templates to recover the observability required.

\subsection{Kalman-like observer}

State-affine systems have the major advantage to give access to Kalman-like observers, where innovation is weighted according to a dynamic symmetric positive-definite gain matrix adapted to time-varying linear systems.
For a given locally essentially bounded input $u$ and state trajectory $x$ of \eqref{syst} with output $y=C(u)x$,
the observer  estimate
$\xhat\in \R^n$ of $x$
follows the dynamical system with gain matrix $S>0$:
\begin{equation}\label{E:observer}
\dot \xhat
= A(u)\xhat + b(u) - S^{-1}C(u)'(C(u)\xhat-y).
\end{equation}
In this paper, we focus on Kalman-like observers where the matrix $S$ satisfies the Lyapunov matrix differential equation (i.e., with a linear
term in place of the usual quadratic one):
\begin{equation}\label{E:observerGain}
      \dot S= -A(u)'S - SA(u) - \theta S + C(u)'C(u).
\end{equation}
The constant
$\theta > 0$ is a gain parameter that we can adjust to achieve tunable speed of convergence under observability.
Although further details  are discussed in \cite{brivadis:hal-03845706},
let us restate some key facts about this observer.
Denoting by $\varepsilon=\xhat-x$ the observer error
governed by the dynamics
$
\dot \eps = \big(A(u) - S^{-1}C(u)'C(u)\big)\eps,
$
the Kalman-like observer inherently possesses its own Lyapunov function
\(
\eps'S\eps
\)
which satisfies
$(\eps'S\eps)(t) \leq \e^{-\theta(t-t_0)}(\eps'S\eps)(t_{0})$ for all $t\geq t_0$,
leading to the fundamental error bound
\begin{equation}\label{E:error_bound}
    |\eps(t)| \leq \e^{-\frac{\theta}{2}(t-t_0)} \sqrt{\frac{S_{\max}(t_{0})}{S_{\min}(t)}}|\eps(t_{0})|,
\end{equation}
with $S_{\min}$ and $S_{\max}$ respectively denoting the smallest and largest eigenvalues of the positive-definite matrix $S$.

The Lyapunov differential equation
\eqref{E:observerGain} with initial
value  condition
$S(t_0)>0$ admits
 the
explicit variation of constants expression
(see, e.g., \cite[Corollary 1.1.6]{MR1997753}).
\begin{multline} \label{E:VariationConst}
S(t)=
\e^{-\theta(t-t_0)}\Phi_u(t_0,t)'S(t_0)\Phi_u(t_0,t)
+
\int_{t_0}^{t} \e^{-\theta(t-\tau)}\Phi_u(\tau, t)'C(u(\tau))'C(u(\tau))\Phi_u(\tau, t)\diff \tau.
\end{multline}
This implies $S(t)>0$ for all $t\geq t_0$, and for any $t_1 \in[t_0, t]$,
\begin{equation}
\label{E:bound_below_S_obs}
    S_{\min}(t)\geq\e^{-\theta(t_1 -t_0)} \lambda_{\min}(\Gram_u(t_1, t_0)),
\end{equation}
where $\lambda_{\min}(\Gram_u(t_1, t_0))$ is the
the smallest eigenvalue of $\Gram_u(t_1, t_0)$.
Therefore, bounding {$S_{\min}$} from below using equation \eqref{E:bound_below_S_obs} directly translates observability into convergence of the observer.
On the other hand, inobservability becomes a critical difficulty when assessing the dynamics of $S$ that we mitigate through control templates.

\section{Output feedback with control templates}
\subsection{Control templates}
\label{sec:templates}

In \cite{andrieu:hal-04387614},
a new approach was developed to overcome observability singularities in output feedback stabilization.
The main tool introduced was the notion of control template, that we recall below in the specific case of state-affine systems.

\smallskip
\begin{definition}
An input $v\in L^\infty([0, \dt],\R^p)$ is a \emph{control template}
if for all $\lambdamax\ge0$ there exists $g(\dt)>0$
such that $\Gram_{\mu\RR v}{(\dt, 0)}\geq g(\dt)\Id$
for all
$(\mu, \RR) \in [0, \lambdamax]\times \O(p)$.
\end{definition}

\smallskip
Note that in \cite{andrieu:hal-04387614}, only analytic control templates were considered.
This had significant implications.
An analytic input $v$ that renders the system observable within a positive time $T$,
makes it observable within any positive time $\dt \le T$.
Furthermore, letting  $\lip{v}$ be the Lipschitz constant of $v$ over the interval $[0, T]$,
$|v(s)-v(0)|\leq \lip{v} \dt$, for all $s\in[0, \Delta]$.

These two advantageous properties, which were
inherently granted for free in \cite{andrieu:hal-04387614} due to analyticity, do not necessarily hold for less regular inputs.
To overcome these limitations,
we introduce the concept of control template families,
bearing in mind that if $v$ is an analytic control template as defined in \cite{andrieu:hal-04387614}, then 
$(v|_{[0, \dt]})_{\dt\in(0, T]}$
constitutes a control template family.
We recall that a function $\alpha:\R_+\to\R_+$ is of class $\KK$ if it is continuous, increasing, and $\alpha(0) = 0$.

\smallskip
\begin{definition}
\label{def:control-template-family}
Let $T>0$ and $\kappa$ be a class $\KK$ function.
A family $(v_\dt)_{\dt\in(0, T]}$ is called a \emph{sampling family} if for
all $\dt$:
\begin{itemize}
    \item $v_\dt\in L^\infty([0, \dt],\R^p)$,
    
    \item $v_\dt$ is continuous at $0$ and $v_\dt(0) = (1, 0, \dots, 0)$,
    \item $|v_{\dt}(s)-v_{\dt}(0)|\leq \kappa(\dt)$, for all $s\in[0, \dt]$.
\end{itemize}
If, in addition, $v_\dt$ is a control template for each $\dt>0$,
we say that $(v_\dt)_{\dt\in(0, T]}$ is a \emph{control template family}.
\end{definition}

\smallskip
Control template families always exist (and are generic in some sense) as a consequence of the universality theorem \cite[Theorem 7]{andrieu:hal-04387614}.
However, this theorem is purely qualitative, and its proof does not furnish any explicit construction.
In the present paper, we propose an explicit construction of such families
(made of piecewise constant inputs)
in Section~\ref{sec:control-template-construction}.

\begin{remark}
In Definition~\ref{def:control-template-family}, we could have chosen $\kappa$ to be linear. However, we have opted to maintain the class $\mathcal{K}$ assumption, which is the appropriate one for our theorems.
\end{remark}

\subsection{Output feedback design}
\label{sec:output-feedback-design}

For any $v\in\R^p$, define the nonempty set
\(
\RRR(v)
= \left\{\RR\in\O(p) \mid \RR 
\left(
|v|, 0, \cdots, 0
\right) =v\right\}.
\)
Following \cite{andrieu:hal-04387614},
we propose a dynamic output feedback design based on the following hybrid dynamics
structure,
with input $u = \mu \RR v_\dt(s)$:
\begin{equation}\label{eq:closed-loop}
\left\{
\begin{aligned}
&\left.\begin{aligned}
&\dot x = A(u)x + b(u)
\\
&\dot \xhat = A(u)\xhat + b(u) 
- S^{-1}C(u)'C(u)(\xhat - x)
\\
&\hspace{-1.5pt}\dot S = -A(u)'S - SA(u) - \theta S
+ C(u)'C(u)
\\
&\dot s = 1
,\qquad\dot \mu = 0
,\qquad\dot \RR = 0
\end{aligned}\hspace{+0.5pt}\right]
\!s\in[0, \dt],
\\
&\left.\begin{aligned}
&x^+ = x, \qquad 
\xhat^+ = \xhat, \qquad
S^+ = S
\\
&\hspace{.3mm}s^+ = 0, \quad\hspace{.4mm}
\mu^+ = |\lambda(\xhat)|, \quad\hspace{.4mm}
\RR^+ \in \RRR(\lambda(\xhat))
\end{aligned}\hspace{0.34cm}\right]
\!s = \dt.%
\end{aligned}\right.
\end{equation}

Let us briefly describe the above hybrid output feedback dynamics.
For any variable $z$ of system~$\eqref{eq:closed-loop}$, $z^+$ means the value after a jump.
The jump times are periodically triggered whenever the timer $s$ reaches $\dt$.
During each interval of length $\dt$, the control law applied to the system is $u=\mu \RR v_\dt(s)$, ensuring observability per the definition of control templates.

At each jump,
the scaling parameter $\mu$ and the isometry $\RR$ are updated such that at the beginning of each time period, $\mu\RR v_\dt(0) = \lambda(\xhat)$.
Choosing $\dt$ small enough ensures that the input remains close to $\lambda(\xhat)$,
while the scaling parameter $\mu$ guarantees that $u\to0$ when $\xhat\to0$.

Let $S_\infty$ be the unique solution of $A'(0)S + S A(0) + \theta S = C'(0)C(0)$.
The set
$\{0\}\times\{0\}\times \{S_\infty\}\times[0, \dt]\times\{0\}\times\O(p)$ is said
to be LES
with basin of attraction containing $\mathcal{B}$
if there exist $\gamma, \omega>0$ such that any solution of \eqref{eq:closed-loop}
with initial condition
$(x_0, \xhat_0, S_0, s_0, \mu_0, \RR_0)\in \mathcal{B}$
satisfies
    $\|(x,\xhat, S-S_\infty,\mu)(t)\| \leq \gamma\e^{-\omega t}\|(x_0,\xhat_0,S_0-S_\infty,\mu_0)\|$.

\smallskip
Our main theorem (proved in Section \ref{sec:proof}) is the following.

\smallskip
\begin{theorem}\label{th:main}
    Suppose Assumptions \ref{ass:stab} and \ref{ass:0obs} hold and let $(v_\dt)_{\dt\in(0, T]}$ be a control template family.
    Then,
    for any compact subset $\mathcal{L} \subset \DD\times\R^n\times \sym$,
    there exist $\dt^*, \lambdamax>0$ such that
    for all $\dt<\dt^*$, there exists $\theta^*>0$,
    such that for all $\theta>\theta^*$
    the set $\{0\}\times\{0\}\times \{S_\infty\}\times[0, \dt]\times\{0\}\times\O(p)$ is
LES for system~\eqref{eq:closed-loop}.
Moreover,
its basin of attraction contains $\mathcal{L}\times[0, \dt]\times[0, \lambdamax]\times\O(p)$.
\end{theorem}

\smallskip
The next section is the other main contribution of this paper. By providing an algorithm for the explicit construction of control template families, it allows to apply Theorem~\ref{th:main} in practice. The section is not mandatory to understand the theorem's proof, which we postpone to Section~\ref{sec:proof}.

\section{Construction of control template families}

We build control template families focusing on piecewise constant templates. In the single-input single-output (SISO) case, a constant input is observable if $\det \Obs(u)\neq 0$ (Kalman rank condition, see Section \ref{SS:obs}).
Since $\det \Obs(u)$ is a polynomial in $u\in\R$ of some degree $d$,
for any collection of 
$N>d$ distinct values
$\{v_k\}_{0\le k<N}\subset \R$
at least one $k$ is such that  $\det \Obs(v_k)\neq 0$.
A piecewise constant function that occupies at least
$N$
distinct values over $[0,\dt]$ will retain this property when scaled. We follow this idea to create a control template.
Beyond the SISO case,  is it possible to pick a finite number
of distinct points
$\{v_k\}_{0\le k<N}\subset \R^p$
that can be rotated and scaled while never all simultaneously lying within the zero locus of any polynomial of degree $d$ or less? The answer is yes, and such points are called in \emph{$(d,p)$-general position}. We provide an algorithm for their construction and the design of a control template family from their knowledge.

\subsection{ SISO case }

Assume that $m=p=1$.
For any $\dt>0$,
define
\begin{equation}\label{E:template-families-siso}
   v_\dt(s) =1 + \frac{\dt}{N}\Big\lfloor \frac{N}{\dt} s\Big\rfloor,
\quad s\in[0, \dt],
\end{equation}
where $\lfloor \cdot \rfloor$ denotes the usual  floor function.
\begin{theorem}
\label{thm:siso-case}
    Assume that Assumption \ref{ass:0obs} holds.
    If $N > \deg \big( \det \Obs(u) \big)$,
    then, for all positive $T$, $(v_\dt)_{\dt\in(0, T]}$ defined by \eqref{E:template-families-siso}
    is a control template family.
\end{theorem}

\begin{proof}
Let $(\mu, \RR)\in [0, \bar\lambda]\times \{-1,1\}$ and set $v = \mu\RR v_\dt$.
For $0\le k \le N$, set $s_k=k\dt/N$.
By definition of $v_\dt$, $v|_{[s_k, s_{k+1})}$ is constant
and $v_\dt(s_k)\neq v_\dt(s_\ell)$ if $k\neq \ell$.
Using \eqref{eq:gram_chasles} yields, for any $k< N$,
\begin{align*}
\Gram_{v}(\dt, 0)
\ge \Phi_{v}\left( s_{k+1}, \dt \right)'\Gram_{v}\left( s_{k+1}, s_{k} \right)\Phi_{v}\left( s_{k+1}, \dt\right).
\end{align*}
Moreover, if $\det \Obs(u)$ were the zero polynomial, it would contradict Assumption \ref{ass:0obs}.
Hence, it admits at most $\deg (\det \Obs(u))$ 
distinct roots.
Thus,
$\det \Obs(v(s_k))\neq 0$ for some $k$.
Because $\Phi_{v}(s_{k+1}, \dt)$ is invertible,
\(
\Gram_{v}(\dt, 0)
\)
is positive.
Note that for all $(\mu, \RR)\in[0, \lambdamax]\times \{-1, 1\}$,
$v$ belongs to the $N$-dimensional subspace of $L^\infty([0, \dt], \R^p)$ of 
functions  that are constant over each subinterval $[s_k, s_{k+1})$.
Consequently,
\(
\Gram_{v}(\dt, 0)
\geq g(\dt)\Id
\)
with
\(
g(\dt)
= \inf_{(\mu, \RR)\in[0, \lambdamax]\times \{-1, 1\}} \lambda_{\min}(\Gram_{\mu \RR v_\dt}(\dt, 0))
\)
being positive because $[0, \lambdamax]\times \{-1, 1\}$ is compact and  
$v\mapsto \Gram_v(\dt, 0)$ is continuous on this $N$-dimensional subspace.
Finally, by construction,
$v_\dt$ is continuous at $0$, $v_\dt(0)=1$ and
\(
|v_\dt(s) -v_\dt(0)|
\le\dt
\)
for any $s\in[0, \dt]$,
which ends the proof.
\end{proof}

\smallskip
When $m,p>1$, an analogue of Theorem~\ref{thm:siso-case} still holds but the whole construction is more involved.
For $m>1$,
$\det(O(u))$ will be replaced
by the determinant of a full-rank $n\times n$ submatrix $M(u)$.
Such a submatrix exists by Assumption~\ref{ass:0obs}, and an upper bound of its degree is
\[
\deg \big( \det M(u) \big)
\le n\deg C +\frac{n(n -1)}{2}\deg A.
\]

\subsection{Construction of points in general position}
\label{sec:general-position}

In this section,
let $d, p\in\N$ and let
{$\card{\cdot}$} denote cardinality.
\begin{definition}
A subset $E\subset \R^p$ is said in $(d, p)$-general position if 
for each polynomial $h\in\R[X_1, \dots, X_p]$ of degree $d$ or less,
there exists $v\in E$ such that $h(v)\neq 0$.
\end{definition}

\smallskip
Only a finite collection of points in $\R^p$ is necessary to get the $(d, p)$-general position property.
If $p=1$, then it is necessary and sufficient to consider $d+1$ distinct points.
If not, the construction becomes more intricate (involving some algebraic considerations).
The complete answer is provided in the following theorem, inspired by the idea from
\cite{benyaacov-hal-2017}.

\smallskip
\begin{theorem}
Let $I = \{1,\ldots,p+d\}$, and assume that $a_1,\ldots,a_{p+d}$ are distinct real numbers.
Let $\Sigma$ denote the collection of all subsets of $I$ of size $p$.
For $\sigma \in \Sigma$, define a polynomial
\begin{gather*}
    f_\sigma(T) = \prod_{i \in \sigma} \, (T - a_i) = T^p + v_{\sigma,1} T^{p-1} + \cdots + v_{\sigma,p}.
\end{gather*}
Let $v_\sigma = (v_{\sigma,1},\ldots,v_{\sigma,p}) \in \R^p$ be the vector of coefficients.
Then $E = \{ v_\sigma : \sigma \in \Sigma \}$
is in $(d, p)$-general position, and $E$ is minimal such.
\end{theorem}

\begin{proof}
\newcommand{\form}{\omega}
Let $X = (X_1,\ldots,X_p)$ denote an indeterminate point in $\R^p$.
For $1 \leq i \leq p+d$, define an affine function $\form_i\colon \R^p \rightarrow \R$ by
\(
    \form_i(X) = a_i^p + a_i^{p-1} X_1 + \cdots + X_p.
\)
Then
\begin{gather*}
    \form_i(v_\sigma) = a_i^p + a_i^{p-1} v_{\sigma,1} + \cdots + v_{\sigma,p} = f_\sigma(a_i).
\end{gather*}
Let $\Xi = \{\xi \in \N^p : \xi_1 + \cdots + \xi_p \leq d\}$.
The monomials in $X$ of degree at most $d$ can be enumerated as
$X^\xi = X_1^{\xi_1}\cdots X_p^{\xi_p}$, for $\xi \in \Xi$,
and a polynomial function $h \colon \R^p \rightarrow \R$ of degree at most $d$ can be written as $\sum_{\xi \in \Xi} h_\xi X^\xi$.
If $\sigma \in \Sigma$, then $\card{I \setminus \sigma} = d$, and we may define such a polynomial function
\begin{gather*}
    g_\sigma(X) = \prod_{i \in I \setminus \sigma} \form_i(X) = \sum_{\xi \in \Xi} w_{\sigma,\xi} X^\xi.
\end{gather*}

Let us define two matrices $V = (v_\sigma^\xi)$ and $W = (w_{\sigma,\xi})$, whose rows are indexed by $\xi \in \Xi$ and columns by $\sigma \in \Sigma$.
Let $D = W' V = (d_{\sigma,\tau})$, where $\sigma,\tau \in \Sigma$.
Then
\begin{align*}
    d_{\sigma,\tau}
    &
    = \sum_{\xi \in \Xi} w_{\sigma,\xi} v_\tau^\xi
    = g_\sigma(v_\tau)
    = \prod_{i \in I \setminus \sigma} \form_i(v_\tau)
    \\ &
    = \prod_{i \in I \setminus \sigma} f_\tau(a_i)
    = \prod_{i \in I \setminus \sigma} \prod_{j \in \tau} \, (a_j - a_i).
\end{align*}
Therefore, $d_{\sigma,\tau} \neq 0$
if and only if $\tau \cap (I \setminus \sigma) = \emptyset $, that is $\tau \subseteq \sigma$.
Since $\tau$ and $\sigma$ have same cardinality $p$, $d_{\sigma,\tau} \neq 0$ if and only if $\sigma = \tau$.
Thus, $D$ is diagonal and $\det D \neq 0$.

It is a standard combinatorial fact that $\card{\Xi} = \binom{p+d}{p}$, so $V$ and $W$ are square matrices, and since $\det D \neq 0$, they are invertible.
Let $h = \sum_{\xi \in \Xi} h_\xi X^\xi$ be a non-zero polynomial of degree at most $d$,
and let $H = (h_\xi)_{\xi \in \Xi}$ be the row matrix of its coefficients.
Since $V$ is invertible,
$H V = \bigl( h(v_\sigma) \bigr)_{\sigma \in \Sigma}$
is non-zero, so $h(v_\sigma) \neq 0$ for some $v_\sigma \in E$.
  
Replacing $E$ with a proper sub-family $E_1$ would replace $V$ with a matrix $V_1$ of rank strictly less than $ \binom{p+d}{p}$.
Any non-zero polynomial whose coefficient vector is in the kernel of $V_1$ would vanish on all points of $E_1$, whence minimality.
\end{proof}

\subsection{Control template family from points in general position}
\label{sec:control-template-construction}

From now on we set
$d =n\deg C + \big(n(n -1)\deg A \big) /2$.
We produce a control sampling family
$(v_\dt)_{\dt\in(0, T]}$
in the following manner.
Let $\{v_{k}\}_{0\le k<N}\subset\R^p$ be in $(d, p)$-general position
(implying $N \ge\binom{p+d}{p}$).
Because being in $(d, p)$-general position is invariant under affine transformations
(as the reader may check),
we assume without loss of generality that $v_0=(1,0,\dots,0)$.
For any $\dt>0$, any $0\le k <N$,
set
\begin{equation}
\label{E:v_Delta-mimo}
v_{\dt}(s)
= v_0 +\dt(v_k -v_0),
\quad s\in\bigg[ \frac{k\dt}{N}, \frac{(k+1)\dt}{N}\bigg).
\end{equation}
Because the \(v_i\)'s are in general position, they do not lie in the zero locus of any minor determinant of the observation matrix.
Therefore,
we have the following theorem, the proof of which proceeds just like that of Theorem~\ref{thm:siso-case}.

\smallskip
\begin{theorem}
\label{thm:mimo-case}
Under Assumption~\ref{ass:0obs},
if $N\ge \binom{p+d}{p}$,
then $(v_\dt)_{\dt\in(0, T]}$ defined by \eqref{E:v_Delta-mimo} is a control template family.
\end{theorem}

\section{Separation principle}\label{sec:proof}

This section is dedicated to the proof of Theorem~\ref{th:main},
following the guidelines of \cite{andrieu:hal-04387614},
where control templates were introduced.
For this reason,
we opt for brevity in computational details.
The proof is organized in three main steps, after some necessary preliminaries. These steps are classical machinery: 
exponential stability of the observer,
completeness of trajectories, and finally closed-loop stabilization.

{\sc Step~0}: \textit{Preliminaries.}
Let us briefly recall the notion of hybrid solutions of \eqref{eq:closed-loop} in the framework of \cite{goedel2012hybrid}.
Clearly, \eqref{eq:closed-loop} satisfies the hybrid basic conditions \cite[Assumption 6.5]{goedel2012hybrid}.
The jump times $(\tau_i)_{i\in\N}$ of \eqref{eq:closed-loop},
determined by the autonomous hybrid subdynamics of the state $s$,
are explicitly given by $\tau_i = \dt-s(0)+i\dt$ for $i\in\N$.
Therefore,
any solution to \eqref{eq:closed-loop}
is defined on a hybrid time domain $E\subset\R_+\times\N$
of the form
$E = \cup_{i=-1}^{I-1}[0, T_e)\cap[\tau_i, \tau_{i+1}]\times \{i\}$
(with $\tau_{-1}\defeq 0$)
where either $T_e =I =+\infty$ (complete trajectory) or
$I\in\N$, and $\tau_{I-1} < T_e \leq \tau_{I}$ (non-complete trajectory).
Because
for any \( (\tau_{i}, i)\in E\),
\(z^+(\tau_{i}, i)=\lim_{t\downarrow \tau_i}z(t, i)\),
where \( z \) denotes any variable of system~$\eqref{eq:closed-loop}$,
we can use \( t \) as a shorthand notation for any \( (t,i) \) without confusion.
Due to the Cauchy-Lipschitz theorem, the Cauchy problem associated to \eqref{eq:closed-loop} admits a unique maximal solution, which is complete if it remains bounded.

For any $u\in\R^p$, put $f(x, u)= A(u)x+b(u)$.
Let $\mathcal{D}$ be the basin of attraction of the origin for the vector field $f(\cdot, \lambda(\cdot))$.
Since we aim to prove semi-global stabilization within $\mathcal{D}$, we select an arbitrary compact subset (for state initial conditions)
$\KK \subset \mathcal{D}$
and design a hybrid dynamic output feedback with basin of attraction containing \(\mathcal{K}\).
According to
~\cite[Théorème 2.348 and Remarque 2.350]{PralyBresch2022a},
there exists a proper function $V\in C^\infty(\mathcal{D}, \R_+)$
such that for any compact $L\subset\DD$,
three positive constants $(\const_i)_{1\leq i\leq 3}$
exist, satisfying
\begin{align}\label{E:lyap-3}
&
\frac{\partial V}{\partial x}(\xi)f\big(\xi, \lambda(\xi) \big)
\leq -\const_1 |\xi|^2,
\\
&
|\xi|^2
\leq V(\xi)\leq \const_2 |\xi|^2,
\qquad
\left|\frac{\partial V}{\partial x}(\xi)\right| 
\leq \const_3 |\xi|,
\nonumber
\end{align}
for all $\xi\in L$.
For any $R\geq 0$,
$\mathcal{D}(R)\defeq\{x\in\R^n\mid V(x)\leq R\}$
is compact because $V$ is proper,
and $\DD(R)\subset \mathring{\DD}(R')$ as soon as $R'>R$.
Now, let $R>0$ be such that $\KK\subset \DD(R)$, and choose $L=\DD(R')$ with $R'>R$,
which fixes the $\const_i$s.
Let $\sat\in C^\infty(\R^n, \R^n)$ be of compact support
and such that
$\sat|_{\DD(R')} =\Id$.
Let $(v_\Delta)_{\Delta\in(0, T]}$ be a control template family 
(such a family exists according to Theorem~\ref{thm:mimo-case}).
From now on,
$\KK$, $R$, $R'$, $\sat$ and$(v_\dt)_{\dt\in(0, T]}$ are fixed.
Set $\lambdamax=\sup_{\DD(R')}|\lambda|$.
Let
$\lip{\lambda_{\sat}}$
and $\lip{f}$
denote the Lipschitz constants of $\lambda\circ\sat$
and $f$ over the compacts of interest, respectively.

We work with the system resulting from substituting the feedback $\lambda$ by the saturated feedback
$\lambda\circ\sat$ to system~\eqref{eq:closed-loop}.
To lighten the presentation we do not explicitly write the $\sat$ function, but all our computations take it into account.

Let $\mathcal{L} \subset \DD\times\R^n\times \sym$ be a compact set whose projection onto $\DD$ is contained in $\KK$
(since $\mathcal{K}$ is arbitrary, $\mathcal{L}$ is as well).
Also, let
\(
(x_0, \xhat_0, S_0, s_0, \mu_0, \RR_0)
\in \mathcal{L}\times [0, \dt] \times[0, \lambdamax]\times\O(p)
\)
and $(x,\xhat, S,s,\mu,\RR)$
be the corresponding maximal trajectory of \eqref{eq:closed-loop}
and denote by $T_e$ its time of existence.
Recall that $\eps = \xhat-x$ denotes the observation error.

{\sc Step~1}: {\it Exponential  stability of the observation error.}
Inequality \eqref{E:bound_below_S_obs} yields, for all $t\in [2\dt, T_e)$, the lower bound
\begin{equation}\label{E:Smin-lower-bound}
S_{\min}(t)\geq \e^{-\theta \dt}g(\dt),
\end{equation}
since $v_\dt$ is a control template.
Then, \eqref{E:error_bound} and \eqref{E:Smin-lower-bound}
imply, for all $t\in [2\dt, T_e)$,
\[
|\varepsilon(t)|\leq \e^{-\frac{\theta}{2}(t -2\dt)}\sqrt{S_{\max}(\tau_0)/g(\dt)}|\eps(\tau_0)|,
\]
which yields
the exponential convergence of $\eps$ towards $0$ with a rate that can be tuned with $\theta$ from time $2\dt$.
Finally, exploiting
the first term of the right handside of 
\eqref{E:VariationConst},
we obtain
that there exists $\gamma(\dt)>0$ independent of the initial conditions such that for all $t\in[0, T_e)$:
\begin{equation}
\label{NEQ:maj-error}
|\eps(t)|
\leq \sqrt{\gamma(\dt)}
\min\big\{\e^{-\frac{\theta}{2}(t -2\dt)}, 1\big\}
|\eps_0|.
\end{equation}

{\sc Step~2}:
{\it The trajectories of system~\eqref{eq:closed-loop} are complete.}
According to \eqref{NEQ:maj-error}, $\err$ is bounded over $[0, T_e)$.
For almost all $t\in[0, T_e)$, we have
(using \eqref{E:lyap-3}):
\begin{align}
\frac{\dd }{\dd t}V(x) +{\const_1}|x|^2
&\le\parfrac{V}{x}(x)
\Big|
f\big( x, {\mu}{\RR} v_\dt(s) \big) -f\big( x, \lambda(x) \big)
\Big|
\nonumber
\\
&\le\const_4 %
    |x| \big| {\mu}{\RR} v_\dt(s)-\lambda(x) \big|,
\nonumber
\end{align}
where $\const_4$ is some constant independent of $\dt$.
According to \cite[Lemma 13]{andrieu:hal-04387614} 
(see also \cite[Appendix D]{lin2019nonlinear} for a result relying on similar computations),
there exists a class $\mathcal{K}$ function $\alpha$ and $\const_5\ge 0$ such that
\(
\big| {\mu}{\RR} v_\dt(s(t)) -\lambda(x(t)) \big|
\leq \alpha(\dt)|x(t)| +\const_5|\err(\tau_i)|
\)
for all $i\ge 0$,
$t \in [\tau_i, \tau_{i+1})$.
It follows, also employing Young's inequality, that, for all $\delta>0$,
\begin{align}
\frac{\dd }{\dd t}V(x) +{\const_1}|x|^2
\le
    \const_4 \big( \alpha(\dt) +\delta\const_5  \big) |x|^2
    +\frac{\const_4\const_5}{\delta} |\err(\tau_i)|^2.
    \label{INEQ:boundedness}
\end{align}
Choose $\delta$ and $\dt^*$ such that
\(
\const_4\big(\alpha(\dt^*) +\delta\const_5  \big) \le{\const_1}.
\)
Then,
taking into account
{inequalities \eqref{E:lyap-3} and \eqref{NEQ:maj-error}},
inequality \eqref{INEQ:boundedness} yields for all $\dt<\dt^*$
{and all $t\ge 3\dt$}
\begin{equation}
\label{NEQ:for-Gronwall-2}
\frac{\dd}{\dd t}V\big( x(t) \big)
\le -\omega V\big (x(t) \big) +\beta \e^{-\theta(t -3\dt)}|\eps_0|^2,
\end{equation}
where we have set
\(
\omega
= \left[{\const_1} -\const_4\big(\alpha(\dt) +\delta\const_5  \big)\right]/\const_2
\)
and
\(
\beta
=
{\gamma(\dt)\const_4\const_5}/\delta.
\)
Inequality \eqref{NEQ:for-Gronwall-2}
implies that
$x$ is bounded,
and because $\err$ were, so is $\xhat$.
Obviously, $s$, $\mu$ and $\RR$ remain bounded, and since $S$ is defined as long as the other coordinates are
(see formula~\eqref{E:VariationConst}), the trajectory is complete (i.e., $T_e=+\infty$).

\smallskip
{\sc Step 3}:
{\it Stabilization of the closed-loop.}
Assume, without loss of generality, that \(\dt^*\) was chosen small enough in Step~2 so that for all initial conditions,
\(
V\big( x(t) \big)
\le {(R+R')}/{2}
\)
for all {$t\in[0, 3\dt]$}.
Applying Grönwall inequality to \eqref{NEQ:for-Gronwall-2} yields,
for all $t\ge 3\dt$ and all $\theta>\omega$,
\begin{align}
V\big( x(t) \big)
&\le \e^{-\omega(t -3\dt)} V\big( x(3\dt) \big)
    +\beta
    \frac{\e^{-\omega(t-3\dt)} -\e^{-\theta(t -3\dt)}}{\theta -\omega}|\eps_0|^2
\label{NEQ:maj-V(t)<R'}
\\
& \le \frac{R +R'}{2} +\frac{\eta}{\theta -\omega},
\nonumber
\end{align}
where $\eta= 2\beta\sup_{(x_0, \xhat_0, S_0)\in\mathcal{L}}|\eps_0|^2$ (recall that $\eps_0=\xhat_0-x_0$).
Let $\theta^*$ be such that
\(
{(R +R')}/{2} +{\eta}/{(\theta^* -\omega)}
\le R'.
\)
If $\theta > \theta^*$, then
$V(x(t))\le R'$ for all $t\ge0$.
Therefore,
the exponential stability of $x$ towards zero follows from \eqref{NEQ:maj-error} and \eqref{NEQ:maj-V(t)<R'}, and the Lipschitz continuity of the flow over $[0, 3\dt]$.
Then, the exponential stability of $S$ towards $S_\infty$ follows, as usual, by choosing
$\theta^*\geq2\sup_{|u|\leq \bar\lambda}\|A(u)\|$
(see e.g. \cite{MR1997753}), and the exponential stability of $\mu$ is obtained by the exponential stability of $x$ and the Lipschitz continuity of $\lambda$.

\section{Numerical simulations}

We propose a numerical implementation of the control template strategy discussed in the paper. We focus on an academic example exhibited in \cite{4753687}. The system is a 2-input, 1-output, 3-dimensional state system, with
\begin{align*}
&A(u_1,u_2)=
\begin{bmatrix}
-0.5-u_1 & 1.5+u_2 & 4 \\
 4.3 & 6 & 5 \\
 3.2 & 6.8 & 7.2 \\
\end{bmatrix},
\\[5pt]
&
b(u_1,u_2)'=
u_1
\begin{bmatrix}
    -0.7&0&0.8
\end{bmatrix}
-u_2
\begin{bmatrix}
    1.3& 4.3& 1.5
\end{bmatrix},
\end{align*}
$$C=\begin{bmatrix}
    1&-0.5&0.5
\end{bmatrix},$$
and the feedback law 
$$
\lambda(x_1,x_2,x_3)
=
\begin{bmatrix}
    0.038 x_1 + 0.1751 x_2 - 0.8551 x_3
    \\
    3.8514 x_1 + 3.84 x_2 + 9.551 x_3
\end{bmatrix}.
$$
Assumptions \ref{ass:pol}, \ref{ass:stab} and \ref{ass:0obs} are satisfied.
Computing the Kalman test yields
$\det(\Obs(u))=(-11.12 - 1.61 u_1 +u_2) (8.84 + 0.16 u_1 + u_2)$.
As such, the singular set $\det(\Obs(u))=0$ is the union of two lines in $\R^2$. Rather than picking points in general position, we can pick a bespoke configuration that is adapted to the given system. Indeed, since the two lines are not perpendicular, no square can have its 4 vertices lying in the singular set. Any piecewise constant control taking its values at the four vertices
of a square cannot be unobservable.
For this reason,
letting $w^1_\dt$ and $w_\dt^2$
be the two indicator functions of $[\dt/2,\dt]$ and $[\dt/4,3\dt/4]$ respectively, the family of inputs $v_\dt=(1,0)+\dt(w^1_\dt,w^2_\dt)$ is a control template family.
A realization of the strategy in this precise case
is shown
in Figure~\ref{fig:simu}, with constants $\theta=50$ and $\dt=0.02$; and initial conditions $x(0)=(2, -2, 3)$,
$\xhat(0)=(-3, 2, -2)$
and $S(0)=\mathrm{Id}$.
\begin{figure}
    \centering
    \includegraphics[width = .95\linewidth]{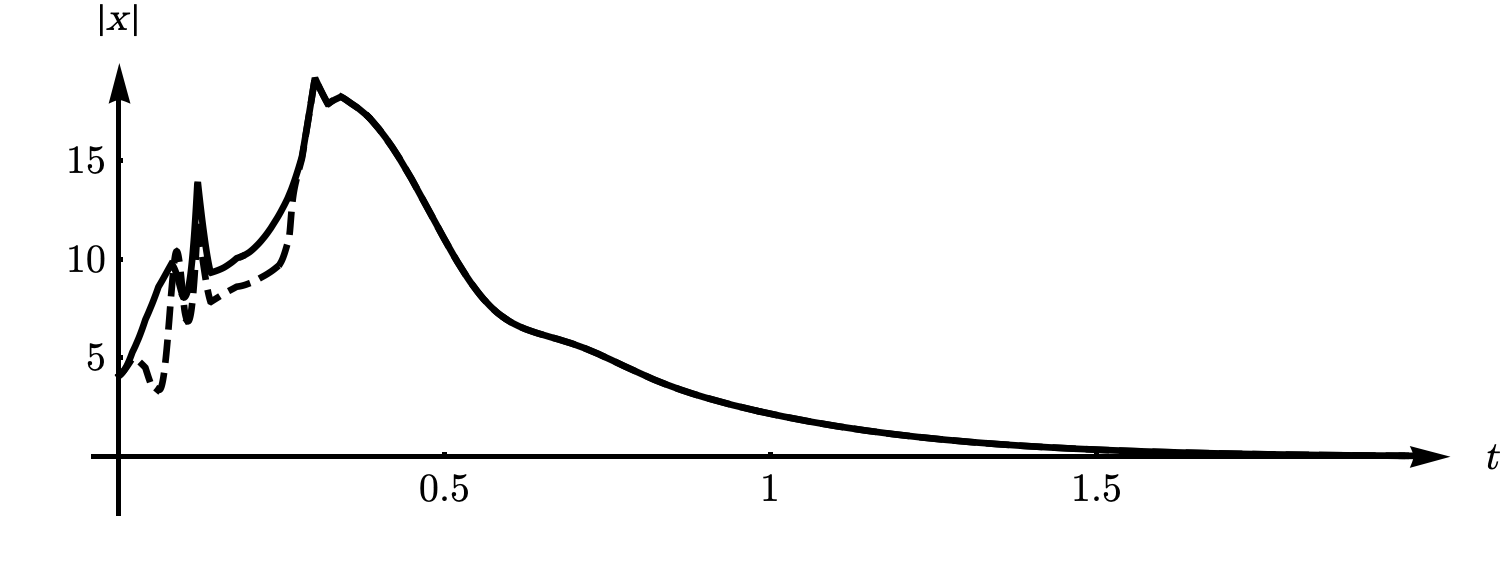}
    \smallskip
    \includegraphics[width = .95\linewidth]{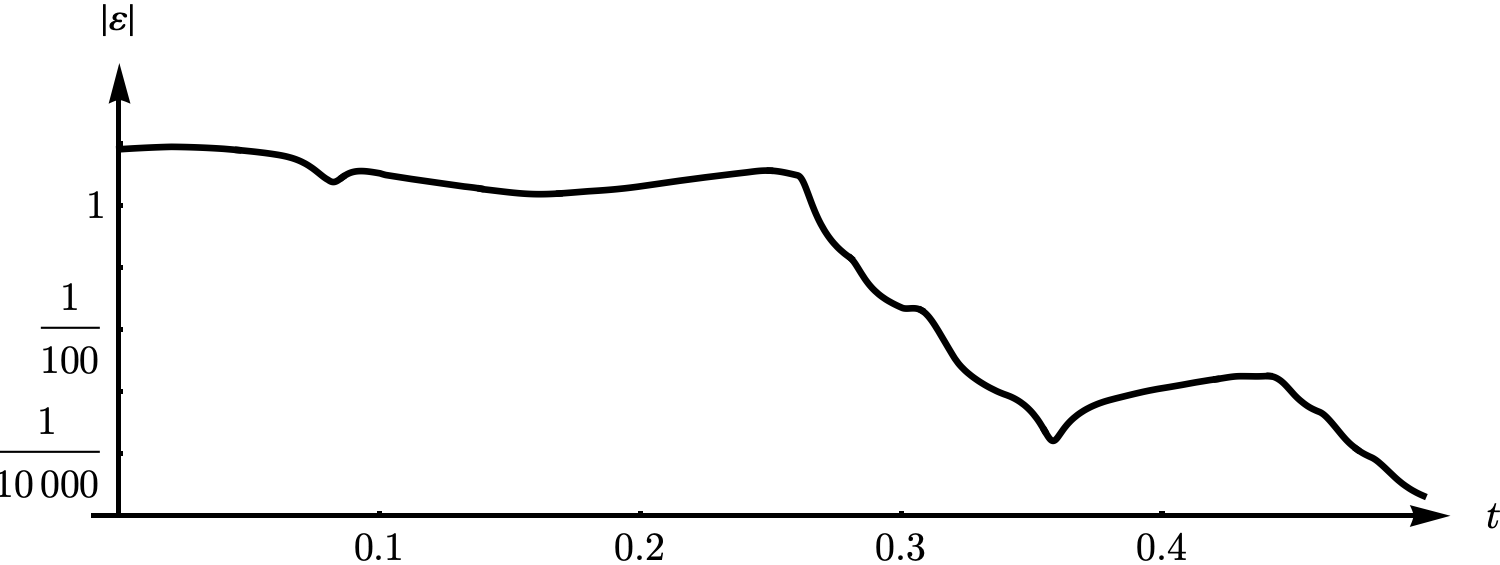}
    \smallskip
    \includegraphics[width = .95\linewidth]{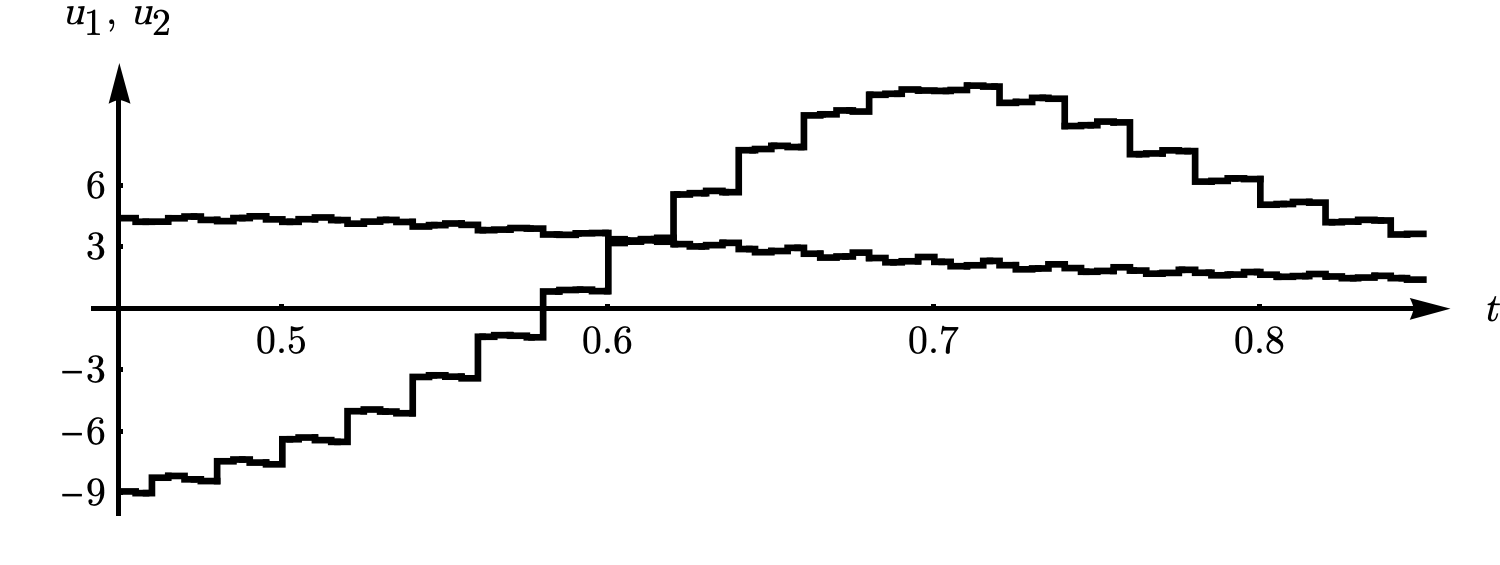}
    
    \caption{Results of the simulation. The top graph shows the evolution of $|x|$ (plain lines) and $|\xhat|$ (dashed lines). The middle one shows the evolution of log error ($\log|\varepsilon|$) of the observer (over the most relevant sub-interval).
    The bottom graph shows the value of the input $(u_1,u_2)$ over a sub-interval as an illustration of the template design.}
    \label{fig:simu}
\end{figure}

\bibliographystyle{abbrv}
\bibliography{references}
\end{document}